\documentclass [12pt]{amsart}
\usepackage{amsmath}%
\usepackage{amsfonts}
\usepackage{amssymb}%
\usepackage{graphicx}
\usepackage{enumerate}

\newtheorem{theorem}{Theorem}[section]

\newtheorem{claim}[theorem]{Claim}

\newtheorem{lemma}[theorem]{Lemma}

\newtheorem*{thm1*}{Zalcman's Lemma}
\newtheorem*{thm2*}{Pang-Zalcman Lemma}
\newtheorem*{thm3*}{Local Pang-Zalcman Lemma}
\newtheorem*{thm4*}{Theorem 1}
\newtheorem*{remark*}{Remark}
\begin{document}
\title{Creating Limit Functions By The Pang-Zalcman Lemma}
\baselineskip 20pt 
\author [Shai Gul and Shahar Nevo]{Shai Gul and Shahar Nevo }
\begin{abstract}
In this paper we calculate the collection of limit functions obtained by applying an extension of Zalcman's Lemma, due to X.C Pang to the non-normal family  $\left\{f(nz):n\in\mathbb{N}\right\}$ in $\mathbb{C}$, where $f=Re^P$. Here $R$ and $P$ are an arbitrary rational function and a polynomial, respectively, where $P$ is a non-constant polnomial.
\end{abstract}
\subjclass[2000]{ 30D20, 30D30, 30D45.}
\keywords {Normal family, Zalcman's Lemma, Spherical metric.}
\thanks{The authors recieved support from the Israel Science Foundation, Grant No. 395/07. This research is part of the European Science Networking Programme HCAA}

\maketitle \tableofcontents

\section {Introduction}

\setcounter{equation}{0} \numberwithin{equation}{section} A
well-known powerful tool in the theory of normal families is the
following lemma of L. Zalcman.
\begin{thm1*} \cite {Za1}
A family $\mathcal{F}$ of functions meromorphic (resp., analytic) on the unit disk $\Delta$ is not normal if and only if there exist \newline
(a) a number $0<r<1$;\newline
(b) points $z_n$, $\left |z_n \right |<r$;\newline
(c) functions $f_n\in \mathcal  F$; and \newline
(d) numbers $\rho_n\rightarrow0^+$, \newline
such that
\begin{align*}
 f_n(z_n+\rho_n\zeta)\mathop\Rightarrow\limits^\chi g(\zeta)\quad (f_n(z_n+\rho_n\zeta)\Rightarrow g(\zeta)) \ ,
\end{align*}
where $g$ is a nonconstant meromorphic (entire) function on $\mathbb {C}$.

Morever, $g$ can be taken to satisfy the normalization \newline $g^\#(\zeta)\leq g^\#(0)=1$, $\zeta\in \mathbb{C}$.
\end {thm1*}
Here and throughout the paper, `$\mathop  \Rightarrow \limits^\chi$' (`$\mathop\Rightarrow$')  means local uniform convergence in $\mathbb{C}$ with respect to the spherical metric (Euclidian metric) of a sequence of meromorphic (holomorphic) functions.

This lemma was generalized by X.C pang as follows.
\begin {thm2*} (\cite[Lemma 2]{Pa1},\cite [Theorem 1] {Pa2})

Given a family $\mathcal F$ of functions meromorphic on the unit disk $\Delta$ which is not normal, then for every $-1<\alpha<1$, there exist \newline
(a) a number $0<r<1$;\newline
(b) points $z_n$, $\left|z_n\right|<r$ for every $n$;\newline
(c) functions $f_n\in \mathcal  F$; and \newline
(d) positive numbers $\rho_n\rightarrow 0^+$, \newline
such that
\begin{align*}
\frac{{f_n (z_n  + \rho _n \zeta )}}{{\rho _n^\alpha  }}\mathop  \Rightarrow \limits^\chi  g(\zeta ) \ ,
\end{align*}
where $g$ is a non-constant function in $\mathbb{C}$.
Morever, $g$ can be taken to satisfy the normalization
$g^\#  \left( \zeta  \right) \le g^\#  \left( 0 \right) = 1$, $\zeta\in\mathbb{C}$.
\end {thm2*}
The case $\alpha=0$ gives Zalcman's Lemma.
These two lemmas have a local version that can be written uniformaly as:
\begin {thm3*} \rm{(LPZ Lemma)} cf. \cite[Lemma 1.5]{Sc}, \cite[Lemma 4.1]{Ne}.

A family $\mathcal  F$ of functions meromorphic in a domain $D\subset\mathbb {C}$ is not normal at $z_0\in D$ if and only if for every $-1<\alpha<1$ there exist\newline
a) points $\left\{z_n\right\}_{n=1}^\infty$, $z_n\rightarrow z_0$;\newline
b) functoins $\left\{f_n\right\}_{n=1}^\infty \in \mathcal F $; \newline
c) positive numbers $\rho_n \rightarrow 0^+$;\newline
such that
\begin {equation}
\rho_n^{-\alpha}f_n(z_n+\rho_n\zeta)\mathop  \Rightarrow \limits^\chi  g(\zeta ) \ ,
\label{eq:1.1}
\end {equation}
where $g$ is a nonconstant meromorphic function on $\mathbb{C}$, such that for every $\zeta\in \mathbb{C}$,
\begin {equation}
g^\#(\zeta)\leq g^\#(0)=1 \ .
 \label {eq:1.1.5}
 \end{equation}
\end{thm3*}
The Pang-Zalcman Lemma and the LPZ Lemma also have extensions in case where we know that the multiplicities of the zeros (or of the poles) of members of the family of functions $\mathcal F$ are large enough (see \cite[Lemma 2]{PZ},\cite[Lemma 3.2]{LN}). In this paper we shall not deal with these extensions, although our particular results are valid also for these extensions.

For a nonconstant function $\mathcal F$ meromorphic on $\mathbb{C}$, let $\mathcal  F(f)$ be the non-normal family in $\mathbb{C}$
 \begin{align*}
 \mathcal  F(f)=\left\{f(nz):n \in \mathbb{N}\right\}.
 \end{align*}
 Normality properties of the family $\mathcal F(f)$ has already been studied from various directions.
 Montel \cite[PP. 158-176]{Mo} was probably the first to deal with this topic. This subject was also studied in \cite{Ne2}, \cite{Ne3} and \cite{GN}.

  The family $\mathcal  F(f)$ is not normal in $\mathbb{C}$, and specifically is never normal at $z=0$.
Given a point $z_0$ where $\mathcal  F(f)$ is not normal and $-1<\alpha<1$, then LPZ Lemma guarantees the existence of at least one function $g(\zeta)$, not constant and meromorphic on $\mathbb{C}$ that is obtained by the convergence process (\ref{eq:1.1}) described in this lemma.
For a certain $-1< \alpha <1$, let $\Pi _\alpha(f)$ denote the collection of \textbf{all} the non-constant limit meromorphic functions $g(\zeta)$ (on $\mathbb{C}$) that are created in the convergence process (\ref{eq:1.1}) (but not necessarily satisfies the normalization (\ref{eq:1.1.5})), considering all the points $z_0 \in \mathbb{C}$ of non-normality of $\mathcal  F(f)$. For such a function $g$, we have by the definition of $\mathcal  F(f)$ and by the LPZ Lemma a sequence $\left\{k_n\right\}_{n=1}^\infty$, $k_n\in\mathbb{N}$, $k_n\rightarrow \infty$, points $z_n\rightarrow z_0$ and positive numbers $\rho_n\rightarrow 0^+$ such that
\begin{equation}
f_{n,\alpha}(\zeta):=\frac{f(k_nz_n+k_n\rho_n\zeta)}{\rho_n^\alpha}\mathop\Rightarrow\limits^\chi g(\zeta) \ .
\label {eq:1.2}
\end {equation}
\medskip

\textbf{Our main goal} in this paper is to calculate, for every $-1<\alpha<1$, the collection $\Pi _\alpha{(f)}$ for the function
\begin {equation}
f(z)=R(z)e^{P(z)},
\label{eq:1.3}
\end{equation}
where $R(z)\not\equiv 0$ is a general rational function and $P(z)$ is a nonconstant polynomial.

Before we state our result we establish some notation: If $z_0$ is a zero (pole) of order $k$ of a nonconstant meromorphic function $f(z)$, then $\tilde{f}_{z_0}(z):=\frac{f(z)}{(z-z_0)^k}$ ($\hat{f}_{z_0}(z):=f(z)(z-z_0)^k$).
Also for $z_0\in\mathbb{C}$ and $r>0$, $\Delta(z_0,r):=\left\{|z-z_0|<r\right\}$, $\overline{\Delta}(z_0,r):=\left\{\left|z-z_0\right|\leq r \right\}$, and for $\theta \in \mathbb{R}$, $R_\theta$ denotes the ray from the origin with argument $\theta$.

Now we state our main theorem.
(The formulation is not short, as the proof is fairly involved.)
\begin{thm4*}
Let $f(z)=R(z)e^{P(z)}$ be as in (\ref{eq:1.3}), where \newline $P(z)=a_k(z-\alpha_1)...(z-\alpha_k)$ (the $\alpha_i$'s may occur with repititions), $a_k\neq 0$; $R(z)=\frac{P_1(z)}{P_2(z)}$ where $P_1(z)=(z-\gamma_1)^{l_1}...(z-\gamma_m)^{l_m}$, $P_2(z)=(z-\beta_1)^{j_1}...(z-\beta_l)^{j_l}$. We assume that $\gamma_1, \cdots ,\gamma_m$; $\beta_1,\cdots,\beta_l$ are all distinct. Let $L_1:=\left|P_1\right|=l_1+...+l_m$, $L_2:=\left|P_2\right|=j_1+...+j_l$.
Then for the various values of $-1<\alpha<1$, $\Pi _\alpha{(f)}$ is given as follows:

    \noindent\textbf{I.} $\boldsymbol{k=\left|P\right|=1}$\newline
    If\quad $\boldsymbol{\alpha=0}$, then
    \vspace {-5 pt}
    \begin{align*}
    {\Pi }_0{(f)}=\left\{k_0e^{A_1\zeta}:k_0\neq 0,\arg{A_1}=\arg{a_1}\right\}\bigcup \\
      \left\{f(C_1+C_2\zeta):C_1\in\mathbb{C},C_2>0\right\}.
    \end{align*}
     \noindent If \quad $\boldsymbol{0<\alpha<1}$, then
    \begin{align*}
\lefteqn{{\Pi }_\alpha(f)=\left\{k_0e^{A_1\zeta}:k_0\neq 0,\arg{A_1}=\arg{a_1}\right\}\bigcup }\\
&\left\{e^{P(\gamma_i)}\tilde{R}_{\gamma_i}(\gamma_i)(A_1\zeta+A_0)^{l_i}:1\leq i \leq m, A_0\in \mathbb{C},A_1>0 \right\}.
    \end{align*}
     \noindent If \quad $\boldsymbol{-1<\alpha<0}$, then
    \begin{align*}
\lefteqn{{\Pi }_\alpha(f)=\left\{k_0e^{A_1\zeta}:k_0\neq 0,\arg{A_1}=\arg{a_1}\right\}\bigcup} \\
& \left\{e^{P(\beta_i)}\hat{R}_{\beta_i}(\beta_i)(A_1\zeta+A_0)^{-j_i}:1\leq i \leq l, A_0\in\mathbb{C}, \ A_1>0\right\}.
    \end{align*}
    \vskip 10 pt
    \noindent\textbf{II.} $\boldsymbol{k \geq 2}$
 If \quad $\boldsymbol{\alpha=0}$, then
    \begin{multline*}
{\Pi }_0{(f)}=\left\{f(C_1+C_2\zeta):C_1\in\mathbb{C},C_2>0 \right\}\bigcup\\
\Big[\bigcup_{l=0}^{k-1}\left\{e^{A_1\zeta+A_0}:A_0\in\mathbb{C},\arg{A_1}=\left( \pm\frac{\pi}{2}(k-1)+\arg{a_k}+(k-1)2\pi l\right)/k \right\}\Big].
\end{multline*}
  If \quad $\boldsymbol{0<\alpha<1}$, then\newline
  for \quad $k=2$
 \begin{multline*}
 {\Pi }_\alpha{(f)}=\Big[\bigcup_{i=1}^m\left\{e^{P(\gamma_i)}A(\zeta+C)^{l_i}:\arg{A}=\arg{\tilde{R}_{\gamma_i}(\gamma_i)},C\in\mathbb{C}\right\}\Big]\bigcup\\
\left\{e^{A_0+A_1\zeta}:A_0\in\mathbb{C},\frac{\pi}{4}+\frac{\arg{a_2}}{2}\leq \arg{A_1} \leq \frac{3\pi}{4}+\frac{\arg{a_2}}{2} \quad \text{or} \quad \right.\\
\left.\frac{5\pi}{4}+\frac{a_2}{2}\leq \arg{A_1} \leq \frac{7 \pi}{4}+\frac{\arg{a_2}}{2}\right\} \ .\\
\end{multline*}
For \quad $ k\geq 3$
\begin{align*}
{\Pi }_\alpha(f)=\Big[\bigcup_{i=1}^{m}\left\{e^{P(\gamma_i)}A(\zeta+C)^{l_i}:\arg{A}=\arg{\tilde{R_{\gamma_i}}}(\gamma_i), C\in\mathbb{C}\right\}\Big]\bigcup
\end{align*}
\begin{align*}
\left\{e^{A_1\zeta+A_0}:A_0\in \mathbb{C},A_1\neq 0 \right\} \ .
\end{align*}
If \quad$\boldsymbol{-1<\alpha<0}$, then \newline
for\quad $k=2$
\begin{align*}
{\Pi }_\alpha(f)=\Big[\bigcup_{i=1}^l\left\{e^{P(\beta_i)}A(\zeta+C)^{-j_i}:\arg{A}=\arg{\hat{R}_{\beta_i}(\beta_i)},C\in \mathbb{C}\right\}\Big]\bigcup
\end{align*}
\begin{align*}
\left\{e^{A_0+A_1\zeta}:A_0\in\mathbb{C},-\frac{\pi}{4}+\frac{\arg{a_2}}{2}\leq \arg{A_1} \leq \frac{\pi}{4}+\frac{\arg{a_2}}{2}\right.
\end{align*}
 \begin{align*}
  \quad \text{or} \quad \left.\frac{3\pi}{4}+\frac{\arg{a_2}}{2}\leq \arg{A_1} \leq \frac{5\pi}{4}+\frac{\arg{a_2}}{2}\right\} \ .
\end{align*}
For \quad $k\geq 3$
\begin{align*}
{\Pi }_\alpha(f)=\Big[\bigcup_{i=1}^l\left\{e^{P(\beta_i)}A(\zeta+C)^{-j_i}: \arg{A}=\arg{\hat{R}_i(\beta_i)},C\in\mathbb{C}\right\}\Big]\bigcup
\end{align*}
\begin{align*}
\left\{e^{A_0+A_1\zeta}:A_0\in\mathbb{C},A_1\neq 0\right\} \ .
\end{align*}
\end{thm4*}
Observe that in each of the three intervals $\alpha=0$, $0<\alpha<1$ and $-1<\alpha<0$, $\Pi _\alpha{(f)}$ is independent of $\alpha$.

The proof of Theorem 1 is similar to climbing a ladder with four steps where each step is more complicated then the former step. In the first step we calculate $\Pi _\alpha{(M)}$ for a general monome, $M(z)=(z-\alpha)^k$. In the second step we find $\Pi _\alpha{(P)}$ where $P$ is a general nonconstant  polynomial. In step 3 we calculate $\Pi _\alpha{(R)}$, where $R$ is a general nonconstant rational function, and finally in the fourth step we find $\Pi _\alpha{(Re^P)}$.
In each step we rely on the results of the previous steps. The first three steps is the contents of section 2, the proof of Theorem 1 is actually the fourth step which we prove in section 3.
We note that for a nonconstant rational function, $z_0=0$ is the only point of non-normality in $\mathbb{C}$, and this is the situation in the first three steps. For $f=Re^P$, the points of non-normality lies on few rays through the origin, as we will see in the sequel. Throughout the proof we often deal with the connections between $\left\{z_n\right\}$ and $\left\{\rho_n\right\}$ in the LPZ Lemma. We hope this will contribute to the better understanding of the potential of this somewhat obscure lemma.
As it is always possible to move to convergent subsequences (in the extended sense), we shall always assume without loss of generality that the sequences $\left\{k_nz_n\right\}$, $\left\{k_n\rho_n\right\}$ from (\ref{eq:1.2}) converge (in the extended sense). This assumption also applies to other sequences of complex numbers involved in our calculations.

The importance of this paper, beyond the result obtained in Theorem~1, lies in the technique that we used. The possible connections between $z_n$ and $\rho_n$ in (\ref{eq:1.1}) were used to deduce the limit function $g$. We note that the Pang-Zalcman Lemma is a common tool to establish normality of families of meromorphic functions. However, the proof of this lemma does not give an explicit relation between $z_n$ to $\rho_n$, because some unknown parameter is involved in this relation (see \cite[Lemma 2]{Pa1}, \cite[Theorem 1]{Pa2}).
Hence, in general there is some difficulty in determing the limit function $g$. We expect that the detailed calculation that given here will contribute and promote the study of this subject.
\section {Calculating $\Pi _\alpha{(M)}$, $\Pi _\alpha{(P)}$ and $\Pi _\alpha{(R)}$ }
\subsection{First step: Calculating $\Pi _\alpha{(M)}$ where $M(z)=(z-\beta)^k$.}
    Let $-1<\alpha<1$ and assume that $M_{n,\alpha}(\zeta)\Rightarrow g(\zeta)$, (where $g$ is a non-constant entire function). This means that
    \begin {equation}
    (k_n\rho_n^{1-\frac{\alpha}{k}}\zeta+\frac{k_nz_n-\beta}{\rho_n^{\frac{\alpha}{k}}})^k\Rightarrow g(\zeta) \ .
    \label{eq:2.1}
    \end{equation}
    The left hand side of (\ref{eq:2.1}) has a single zero of multiplicity $k$ in $\mathbb{C}$, and thus, it follows by Rouch\'e's Theorem that $g(\zeta)$ is also a monome of degree $k$.
    There must be $0<A<\infty$ and $C\in\mathbb{C}$, such that $k_n\rho_n^{1-\frac{\alpha}{k}}\rightarrow A$ and $\frac{k_nz_n-\beta}{\rho_n^{\frac{\alpha}{k}}} \rightarrow C$ and so $g(\zeta)=(A\zeta+C)^k$.
    Conversely, given $A>0$ and $C \in\mathbb{C}$, we set
    \begin{equation}
    k_n=n,\quad \rho_n=(\frac{A}{n})^{\frac{k}{k-\alpha}},\quad z_n=\frac{A^{\frac{\alpha}{k-\alpha}}C+\beta n^{\frac{\alpha}{k-\alpha}}}{n^{1+\frac{\alpha}{k-\alpha}}}
    \label{eq:2.1.5}
    \end{equation}
    to get (for every $n$) $M_n(\zeta)=(A\zeta+C)^k$. Thus, for every $-1<\alpha<1$
    \begin{equation}
    {\Pi }_\alpha{(M)}=\left\{(A\zeta+C)^k: A>0, C\in \mathbb{C}\right\} \ .
    \label {eq:2.1.75}
    \end{equation}
    \newpage
\subsection{Second step: Calculating $\Pi _\alpha(P)$ for a nonconstant polynomial $P(z)$}
Let $P(z)=L(z-\gamma_1)^{l_1}...(z-\gamma_m)^{l_m}$, $\gamma_i\neq \gamma_j$, $i \neq j$, $k:=l_1+l_2...+l_m$.
Assume first that $\alpha=0$ and that
\begin{equation}
P_{n,0}(\zeta)=P(k_n\rho_n\zeta+k_nz_n)\Rightarrow g(\zeta) \ .
\label{eq:2.2}
\end{equation}
By substituting $\zeta=0$ in (\ref{eq:2.2}), we get that $\left\{k_nz_n\right\}$ is bounded and thus $k_nz_n \rightarrow C\in \mathbb{C}$ (recall that we always assume without loss of generality that $\left\{k_nz_n\right\}$, $\left\{k_n\rho_n\right\}$, etc. converge). Now, if $k_n \rho_n \rightarrow 0$ then $g$ is constant and in case that $k_n \rho_n \rightarrow \infty$ then $g(\zeta)=\infty$ for every $\zeta \neq 0$. Hence $k_n\rho_n \rightarrow A$, $0<A<\infty$ and we have $g(\zeta)=P(A \zeta+C)$.

On the other hand, given $0<A<\infty$ and $C \in \mathbb{C}$, the trivial setting $k_n=n$, $\rho_n=\frac{A}{n}$, $z_n=\frac{C}{n}$ gives $P_{n,0}(\zeta)=P(A \zeta+C)$ and we get
\begin{equation}
{\Pi }_0(P)=\left\{P(A\zeta+C): A>0, C \in \mathbb{C} \right\} .
\label{eq:2.2.5}
\end{equation}
Consider now the case where $0<\alpha<1$. Here $P_{n,\alpha}(\zeta) \Rightarrow g(\zeta)$ means
\begin{equation}
\frac{L(k_n\rho_n\zeta+k_nz_n-\gamma_1)^{l_1}...(k_n\rho_n\zeta+k_nz_n-\gamma_m)^{l_m}}{\rho_n^\alpha} \Rightarrow g(\zeta)\ .
\label{eq:2.3}
\end{equation}
Because of $\rho_n^\alpha \rightarrow 0$, then by substituting $\zeta=0$ in (\ref{eq:2.3}), we get that there exists $1\leq i \leq m $  such that $k_nz_n \rightarrow \gamma_i$, since otherwise $P_{n,\alpha}(0) \rightarrow \infty$, and this would be a contradiction.

Without loss of generality, we assume that $i=1$.
\begin{claim}
$k_n\rho_n \rightarrow 0$.
\end {claim}
\begin{proof}
Indeed, if $k_n\rho_n \rightarrow \infty$, then for every $\zeta \neq 0$, $P_{n,\alpha}(\zeta) \rightarrow \infty$, a contradiction.

If $k_n\rho_n \rightarrow A$, $0<A<\infty$, then there are some $R>0$ and $N_0 \in \mathbb{N}$ such that for every $\zeta$, $\left| \zeta \right|>R$, $n>N_0$ and
$1\leq i\leq m$, $\left|k_n\rho_n\zeta+k_nz_n-\gamma_i\right|\geq 1$ and thus $P_{n,\alpha}(\zeta) \rightarrow \infty$, a contradiction and the claim is proved.
\end{proof}
We then get from (\ref{eq:2.3}) that
\begin{align*}
L \frac{(k_nz_n-\gamma_1+k_n\rho_n\zeta)^{l_1}}{\rho_n^\alpha}(\gamma_1-\gamma_2)^{l_2}(\gamma_1-\gamma_3)^{l_3}...(\gamma_1-\gamma_m)^{l_m} \Rightarrow g(\zeta) \  .
\end{align*}
From the result in section 2.1 we then get that $g(\zeta)=\tilde{P}_{\gamma_1}(\gamma_1)(A\zeta+C)^{l_1}$ where $A>0$ and $C \in \mathbb{C}$.

Conversely, given $A>0$ and $C\in\mathbb{C}$, an analogous setting to (\ref{eq:2.1.5})
\begin{center}
 $k_n=n$, $\rho_n=(\frac{A}{n})^{\frac{l_1}{l_1-\alpha}}$, $z_n=\frac{A^{\frac{\alpha}{l_1-\alpha}}C+\gamma_1 n^{\frac{\alpha}{l_1-\alpha}}}{n^{1+\frac{\alpha}{\l_1-\alpha}}}$
 \end{center}
  gives
 \begin{align*}
 P_{n,\alpha}(\zeta) \Rightarrow \tilde{P_{\gamma_1}}(\gamma_1)(A\zeta+C)^{l_1} \ .
 \end{align*}
Observe that since $0<\alpha<1$, indeed $n\rho_n \rightarrow 0$. Running over all the roots $\gamma_i$, $1\leq i\leq m$, of $P(z)$ we get that
\begin{equation}
{\Pi }_\alpha(P)=\left\{\tilde{P}_{\gamma_i}(\gamma_i)(A\zeta+C)^{l_i}: A>0,\ C \in \mathbb{C},\ 1 \leq i \leq m \right\} \ .
\label {eq:2.3.5}
\end{equation}
We turn now to the case $-1<\alpha<0$.
Suppose that
\begin{equation}
P_{n,\alpha}(\zeta) \Rightarrow g(\zeta) \ .
\label{eq:2.4}
\end{equation}
\begin{claim}{\label{Claim 1}}
$k_n\rho_n \rightarrow \infty$.
\end{claim}
\begin{proof}
If to the contrary, $k_n \rho_n \rightarrow A$, $A<\infty$ and $k_nz_n \rightarrow C\in \mathbb{C}$, then $P_{n,\alpha}(\zeta) \rightarrow 0$ for every $\zeta \in \mathbb{C}$ and this is of course a contradiction. If $k_n \rho_n \rightarrow A<\infty$ and $k_nz_n \rightarrow \infty$ then (\ref{eq:2.4}) gives
\begin{equation}
L \frac{{(k_nz_n)}^k}{\rho_n^\alpha}\underbrace{\left[1+\frac{k_n\rho_n\zeta-\gamma_1}{k_nz_n}\right]^{l_1}...\left[1+\frac{k_n\rho_n\zeta-\gamma_m}{k_nz_n}\right]^{l_m}}_{T_n(\zeta)} \Rightarrow g(\zeta) \ .
\label{eq:2.4.5}
\end{equation}
Since
\begin{align*}
T_n(\zeta) \Rightarrow 1 \ ,
\end{align*}
we get that $L \frac{{(k_nz_n)}^k}{\rho_n^\alpha} \Rightarrow g(\zeta)$ and we get that $g$ is a constant, a contradiction.
\end{proof}
\begin{claim}{\label{Claim 2}}
$\frac{z_n}{\rho_n} \rightarrow B \in \mathbb{C}$ (equivalenty, for every $1 \leq i \leq m$, $\frac{k_nz_n-\gamma_i}{k_n\rho_n} \rightarrow B )$.
\end{claim}
\begin{proof}
If this were not the case, then for every $1 \leq i \leq m$, $\frac{k_n\rho_n}{k_nz_n-\gamma_i} \rightarrow 0$, and then
\begin{align*}
 P_{n,\alpha}(\zeta)=\frac{L}{\rho_n^\alpha}\left[\prod  \limits_{i=1}^m(k_nz_n-\gamma_i)^{l_i}\right]\cdot \underbrace{\left[1+\frac{k_n\rho_n}{k_nz_n-\gamma_1}\zeta\right]^{l_1}\cdots\left[1+\frac{k_n\rho_n}{k_nz_n-\gamma_m}\zeta\right]^{l_m}}_{S_n(\zeta)}
 \end{align*}
 \begin{align*}
  \Rightarrow g(\zeta) \  .
 \end{align*}
Here also $S_n(\zeta) \Rightarrow 1$ and as in Claim \ref{Claim 1}, we get a contradiction and Claim \ref{Claim 2} is proven.
\end{proof}
We can write (\ref{eq:2.4}) as
\begin {equation}
L \frac{(k_n\rho_n)^k}{(\rho_n^\alpha)}\underbrace{\left[\zeta+\frac{k_nz_n-\gamma_1}{k_n\rho_n}\right]^{l_1}\cdots\left[\zeta+\frac{k_nz_n-\gamma_m}{k_n\rho_n}\right]^{l_m}}_{R_n(\zeta)}\Rightarrow g(\zeta) \ ,
\label{eq:2.5}
\end{equation}
and since $R_n(\zeta)\Rightarrow(\zeta+B)^k$, we have $\frac{(k_n\rho_n)}{{\rho_n}^{\alpha/k}} \rightarrow A$, $0<A<\infty$.
Thus $g(\zeta)=L(A \zeta+C)^k$, where $C=AB$.

Conversely, let $g(\zeta)=L(A \zeta+C)^k$ where $A>0$, $C\in \mathbb{C}$. We set $k_n=n$ and consider (\ref{eq:2.5}), we wish that $A=\frac{n\rho_n}{\rho_n^{\frac{\alpha}{k}}}$ and $\frac{z_n}{\rho_n}=\frac{C}{A}$. These requirements are fulfiled by the setting $\rho_n:=(\frac{A}{n})^{\frac{k}{k-\alpha}}$, $z_n:=\frac{C}{A}(\frac{A}{n})^{\frac{k}{k-\alpha}}$. Hence we get that for $-1< \alpha <0$
\begin{equation}
{\Pi }_\alpha{(P)}=\left\{L(A \zeta+C)^k: A>0,\ C\in \mathbb{C}\right\} \ .
\label{eq:2.5.5}
\end{equation}
\subsection {Third step: Calculating ${\Pi }_\alpha(R)$ for a rational function $R(z)$}
    \textbf{I.} We assume first that $R$ has at least one zero and one pole in $\mathbb{C}$. Denote
    \begin{equation}
    R(z)=L \frac{(z-\gamma_1)^{l_1}...(z-\gamma_m)^{l_m}}{(z-\beta_1)^{j_1}...(z-\beta_l)^{j_l}},\ k=l_1+...+l_m>0, \ j=j_1+...+j_l>0.
    \label{eq:2.5.75}
    \end{equation}
    We assume that for some $-1<\alpha<1$
    \begin{equation}
    R_{n,\alpha}(\zeta){\stackrel{\chi}{\Rightarrow}} g(\zeta) \ .
    \label{eq:2.6}
    \end{equation}
    Observe first that Picard's great theorem and Rouch\'e's Theorem imply that ${\Pi }_\alpha(R)$ contains only rational functions.
    We separate into subcases according to the value of $\alpha$.\newline
\vskip 10 pt
    \noindent\textbf{Case (A)}: $0<\alpha<1$.

    Let us assume first that $k_n\rho_n \rightarrow C$, $0<C<\infty$.
    In such case, if $k_nz_n \rightarrow \infty$, then as in (\ref{eq:2.4.5}) we deduce that $g$ is a constant, a contradiction. If there exists some $b \in \mathbb{C}$ such that $k_nz_n \rightarrow b$, then by (\ref{eq:2.6}) (observe that $\rho_n^\alpha \rightarrow 0 $) we get for every $0 \leq \theta <2 \pi$, except finitely many $\theta$'s, that $R_{n,\alpha}(\zeta)\rightarrow \infty$ for every $\zeta=re^{i\theta}$, $r>0$. This is a contradiction. \newline
    Secondly, we assume that $k_n\rho_n \rightarrow 0$. In such a situation if $k_nz_n \rightarrow \infty$ then $g(\zeta) \equiv d$ where $d$ is some finite constant or
     $d \equiv \infty$, a contradiction.
     If $k_nz_n \rightarrow \eta$, $\eta \in \mathbb{C}$, then if for every $i$, $j$ $\eta \neq \gamma_i, \beta_j$ then $g \equiv \infty$, a contradiction.

    If $\eta=\beta_{i_0}$ for some $j_0$, $1\leq j_0 \leq l$, then also by (\ref {eq:2.6}) $g\equiv \infty$, a contradiction.

    If $\eta=\gamma_{i_0}$ for some $1\leq i_0 \leq m$, then assume without loss of generality that $\eta=\gamma_1$. Then (\ref{eq:2.6}) can be written as
    \begin{align*}
     \frac{1}{\rho_n^\alpha}(k_nz_n-\gamma_1+k_n\rho_n\zeta)^{l_1}\tilde{R}_{\gamma_1}(k_nz_n+k_n\rho_n\zeta){\stackrel{\chi}{\Rightarrow}} g(\zeta) \ ,
     \end{align*}
     and since
     \begin{align*}
     \tilde{R}_{\gamma_1}(k_nz_n+k_n\rho_n\zeta){\stackrel{\chi}{\Rightarrow}} \tilde{R}_{\gamma_1}(\gamma_1) \ ,
     \end{align*}
we get by the case of a monome that
\begin {equation}
g(\zeta)=\tilde{R}_{\gamma_1}(\gamma_1)(A\zeta+C)^{l_1} \quad  A>0,\ C \in \mathbb{C} \ .
\label{eq:2.7}
\end {equation}
As in section 2.1, it can easily be shown that every function of the form (\ref{eq:2.7}) is in $\Pi _{\alpha}(R)$. Recall now that $C_0$ can be any value $1\leq i_0 \leq m$, and we get that the contribution to $\Pi _{\alpha}(R)$ from this possibility is
\begin{equation}
\left\{g(\zeta)=\tilde{R}_{\gamma_i}(\gamma_i)(A\zeta+C)^{l_i}: \quad C \in \mathbb{C}, \ A>0 ,\ 1 \leq i \leq m \right\}.
\label {eq:2.8}
\end{equation}
The last option in case (A) is that $k_n\rho_n \rightarrow \infty$. Similarly  to the case $k_n\rho_n \rightarrow 0$, we deduce that $\frac{z_n}{\rho_n}\rightarrow C$, $C \in \mathbb{C}$. (Recall that we can assume with no loss of generality that sequences as $\left\{\frac{z_n}{\rho_n}\right\}$ converges in the extended sense.)
We can write
\begin{align*}
R_{n,\alpha}(\zeta)=L\frac{(k_n\rho_n)^{l_1+\dots+l_m}(\zeta+\frac{k_nz_n-\gamma_1}{k_n\rho_n})^{l_1}\dots (\zeta+\frac{k_nz_n-\gamma_m}{k_n\rho_n})^{l_m}}{\rho_n^\alpha(k_n\rho_n)^{j_1+\dots+j_l}(\zeta+\frac{k_nz_n-\beta_1}{k_n\rho_n})^{j_1}\dots(\zeta+\frac{k_nz_n-\beta_l}{k_n\rho_n})^{j_l}}
\end{align*}
\begin{align*}
=L\frac{(k_n\rho_n)^{k}(\zeta+\frac{k_nz_n-\gamma_1}{k_n\rho_n})^{l_1}\dots (\zeta+\frac{k_nz_n-\gamma_m}{k_n\rho_n})^{l_m}}{\rho_n^\alpha(k_n\rho_n)^{j}(\zeta+\frac{k_nz_n-\beta_1}{k_n\rho_n})^{j_1}\dots(\zeta+\frac{k_nz_n-\beta_l}{k_n\rho_n})^{j_l}}
\end{align*}
\begin{align*}
= L\frac{(k_n\rho_n)^{k-j}(\zeta+\frac{k_nz_n-\gamma_1}{k_n\rho_n})^{l_1}\dots (\zeta+\frac{k_nz_n-\gamma_m}{k_n\rho_n})^{l_m}}{\rho_n^\alpha(\zeta+\frac{k_nz_n-\beta_1}{k_n\rho_n})^{j_1}\dots(\zeta+\frac{k_nz_n-\beta_l}{k_n\rho_n})^{j_l}} \ .
\end{align*}
Observe that for every $i$ and $j$, $\frac{k_nz_n-\gamma_i}{k_n\rho_n}, \frac{k_nz_n-\beta_j}{k_n\rho_n} \rightarrow C$.
Thus, if $k \geq j$ this is a contradiction, since the only candidate to be a limit function is $g\equiv \infty$.

If $k<j$, then $L_0:= \lim\frac{(k_n\rho_n)^{k-j}}{{\rho}_n^\alpha}$ must satisfy $L_0\neq 0, \infty$, since otherwise $g\equiv 0$ or $g \equiv \infty$, as the value of $L_0$. We deduce that $g(\zeta)=L\cdot L_0 (\zeta+C)^{k-j}$. But $R_{n,\alpha}(\zeta)$ vanishes at $\frac{k_nz_n-\gamma_1}{k_n\rho_n}, \dots,  \frac{k_nz_n-\gamma_m}{k_n\rho_n}$ and thus $g(-C)=0$, a contradiction. Hence the collection (\ref{eq:2.8}) is $\Pi _{\alpha}(R)$. \newline
\vskip 10 pt
\noindent\textbf{Case (B)}: $-1<\alpha<0$.

The calculation of $\Pi _{\alpha}(R)$ is immediate since $R_{n,\alpha}(\zeta){\stackrel{\chi}{\Rightarrow}} g(\zeta)$ in $\mathbb{C}$ if and only if $(\frac{1}{R})_{n,-\alpha}(\zeta){\stackrel{\chi}{\Rightarrow}} \frac{1}{g}(\zeta)$ in $\mathbb{C}$, and since $0<-\alpha<1$. Thus, by Case~(A), $\Pi _{\alpha}(R)=\left\{\hat{R}_{\beta_n}(\beta_n)((A\zeta+C)^{j_n})^{-1} : A>0,\ C \in \mathbb{C},\ 1 \leq n \leq l \right\}$.
\vskip 10 pt
\textbf{Case (C)}: $\alpha=0$.

Assume first that $k_n \rho_n \rightarrow 0$. Then if $k_nz_n \rightarrow \infty$ we deduce that $g \equiv c$, $c \in \mathbb{C}$, a contradiction.

If $k_nz_n \rightarrow b$, $b \in \mathbb{C}$, then in case $b \neq \gamma_i,\beta_j$ for every $i,j$ we get by (\ref{eq:2.6}) that $g$ is constant, a contradiction.

If $b=\gamma_{i_0}$, $1 \leq i_0 \leq m$, then $g \equiv 0$, a contradiction. If $b=\beta_{j_0}$, $1 \leq j_0 \leq l$ then $g \equiv \infty$, a contradiction.

The next possibility we examine is $k_n \rho_n \rightarrow \infty$. As in Case~(A) or Case~(B) we must have $\frac{z_n}{\rho_n} \rightarrow c \in \mathbb {C}$. Then we can write
\begin{align*}
 R_{n,0}(\zeta)=L(k_n \rho_n)^{k-j} \frac{(\zeta+\frac{k_nz_n-\gamma_1}{k_n \rho_n})^{l_1} \dots (\zeta+\frac{k_nz_n-\gamma_m}{k_n \rho_n})^{l_m}}{(\zeta+\frac{k_nz_n-\beta_1}{k_n \rho_n})^{j_1} \dots (\zeta+\frac{k_nz_n-\beta_l}{k_n \rho_n})^{j_l}} \ .
 \end{align*}
  In any of the cases $k=j$, $k>j$ or $k<j$, we get a contradiction. So it must be the case $k_n \rho_n \rightarrow c$, $0<c <\infty$. Then, if $k_nz_n \rightarrow \infty$ then similarly to the case $k_n \rho_n \rightarrow 0$, we get that $g$ is constant so $k_nz_n \rightarrow b$, $b \in \mathbb{C}$ and $g(\zeta)=R(b+c\zeta)$.

   Conversely, for every $b \in \mathbb{C}$, $c>0$, we can take $k_n=n$, $\rho_n=\frac{c}{n}$, $z_n=\frac{b}{n}$ to get $R_{n,0}(\zeta) \stackrel{\chi}\Rightarrow R(b+c \zeta)$ in $\mathbb{C}$, so $\Pi _0{(R)}=\left\{ R(b+c\zeta): b \in \mathbb{C}, c>0 \right\}$.
   \vskip 10 pt
\noindent\textbf{II.} Now we consider the case where $R(z)$ has only zeros or only poles. If $R(z)$ has only zeros, then $R$ is a polynomial  and this case was discussed in section 2.2. If $R(z)$ has only poles then $R=\frac{1}{P}$ where $P$ is a polynomial, and we can use the same principle as in Case (B) of (I) of the present subsection, and then deduce by the results in section 2.2 (see (\ref{eq:2.2.5}), (\ref{eq:2.3.5}) and (\ref{eq:2.5.5})) the following:

 For $\alpha=0$ we get by (\ref{eq:2.2.5})
\begin{align*}
{\Pi }_0{(R)}=\left\{R(A\zeta+C): A>0, C \in \mathbb{C}\right\}.
\end{align*}
For $0<\alpha<1$ we get by (\ref{eq:2.5.5})
\begin{align*}
 {\Pi }_\alpha{(R)}=\left\{\frac{L}{(A\zeta+C)^j}:A>0, C \in \mathbb{C}\right\} \ .
 \end{align*}
 And for $-1<\alpha<0$ we have by (\ref{eq:2.3.5})
   \begin{align*}
    {\Pi }_\alpha{(R)}=\left\{\hat{R}_{\beta_n}(\beta_n)((A\zeta+C)^{j_n})^{-1}: A>0,\ C \in \mathbb{C},\ 1 \leq n \leq l \right\} \ .
    \end{align*}
\section{Finding $\Pi _\alpha(Re^P)$}
Let $f(z)=R(z)e^{P(z)}$ where
\begin{equation}
\label{eq:3.1}
 R=\frac{P_1}{P_2},\ P_1(z):=(z-\gamma_1)^{l_1}\dots (z-\gamma_m)^{l_m},
 \end{equation}
 \begin{align*}
 P_2(z):=(z-\beta_1)^{j_1}\cdots(z-\beta_l)^{j_l},\ L_1:=\big|P_1\big|=l_1+\cdots+l_m;
 \end{align*}
 \begin{align*}
  L_2:=\big|P_2\big|=j_1+\cdots+j_l  ,\ L_1,L_2 \geq 0 ,\ L\neq 0 \ .
\end{align*}
The case $R=L_1\frac{P_1}{P_2}$, $L\neq 0,1$ is also included here, i.e., we can assume that $L=1$, since otherwise $L=e^{a'_0}$, $a_0'\neq 0$. We can write $\hat{a}_0=a_0+a'_0$ instead of $a_0$ as the constant coefficient of $P(z)$.

Also let us denote $P(z)=a_kz^k+a_{k-1}z^{k-1}+\dots +a_0$, $a_k \neq 0$. We wish to find $\Pi _\alpha(f)$ for $-1<\alpha<1$, but first we need some preparation.
\subsection{Auxiliary lemmas and a remark}
\begin{lemma}
Let $f$ be a nonconstant meromorphic function in $\mathbb{C}$ and $-1<\alpha<1$. Then
\begin{enumerate}
    \item[{\rm(1)}] If $g(\zeta) \in \Pi _\alpha (f)$ then for every $C \in \mathbb{C}$ $g(\zeta+C) \in \Pi _\alpha(f)$\newline
      and
    \item[{\rm(2)}]  If $e^{a\zeta+b} \in \Pi _\alpha(f)$ then for every $a' \neq 0$ such that $\arg{(a')}=\arg{(a)}$ and for every $b' \in \mathbb{C}$, $e^{a'\zeta+b'} \in \Pi _\alpha(f)$.
\end {enumerate}
\end{lemma}
\begin{proof}
Suppose that $g \in \Pi _\alpha{(f)}$, then we have $\frac{f(k_nz_n+k_n\rho_n(\zeta+C))}{\rho_n^\alpha}\stackrel{\chi}{\Rightarrow} g(\zeta)$ in $\mathbb{C}$, with $\rho_n \rightarrow 0^+$, $z_n \rightarrow z_0$ and $k_n \in \mathbb{N}$. We set $\rho'_n=\rho_n$, $z'_n=z_n+\rho_nC \rightarrow z_0$ and get
\begin{align*}
\frac{f(k_nz'_n+k_n\rho'_n\zeta)}{{\rho'}_n^\alpha}=\frac{f(k_nz_n+k_n\rho_n(\zeta+C))}{\rho_n^\alpha} \stackrel{\chi}{\Rightarrow}g(\zeta+C) \ ,
\end{align*}
 and this proves (1). For the proof of (2) assume that \newline $\frac{f(k_nz_n+k_n\rho_n\zeta)}{\rho_n^\alpha}\stackrel{\chi}{\Rightarrow} e^{a\zeta+b}$ in $\mathbb{C}$. Define for $a'$ with $\arg{a'}=\arg{a}$, $\rho'_n=\frac{a'}{a}\rho_n \rightarrow 0^+$ and $(\frac{A'}{A})^{-\alpha}=e^{b_0}$, where $b_0\in\mathbb{R}$. We have
 \begin{align*}
  \frac{f(k_nz_n+k_n\rho'_n \zeta)}{(\rho'_n)^\alpha}=\frac{f(k_nz_n+k_n\rho_n(\frac{a'}{a}\zeta))}{{\rho}_n^\alpha (a'/a)^\alpha} \stackrel{\chi}{\Rightarrow} g(\frac{a'}{a}\zeta)e^{b_0}=e^{a'\zeta+b+b_0} \ .
  \end{align*}
   By (1) we can replace $b+b_0$ with every $b' \in \mathbb{C}$. This completes the proof of the lemma.
\end{proof}
\begin{remark*}
Let $F$ be a family of non-vanishing holomorphic functions which is not normal at $z_0$ and let $-1<\alpha<1$. Then the convergence process (\ref{eq:1.1}) in the LPZ Lemma guarantees a limit function $g(\zeta)$ with $g^{\#}(\zeta)\leq 1$ for every $\zeta \in \mathbb{C}$. By a theorem of Clunie and Hayman \cite[Theorem 3]{CH}, the order of $g$ is at most $1$ and since $g(\zeta)\neq 0$, $\zeta \in \mathbb{C}$, by Hurwitz's Theorem we deduce that $g(\zeta)=e^{a\zeta+b}$. The results which we will prove in the detailed process of calculating $\Pi _\alpha(Re^P)$ are indeed consistent with this theorem of Clunie and Hayman.
\end{remark*}
\begin{lemma}
Let $f=Re^P$ be given by (\ref{eq:3.1}). Then the points where $F(f)$ is not normal in $\mathbb{C}$ are exactly
\begin{equation}
\left\{\bigcup_{l=0}^{k-1}R_{\theta_k^{+}(l)}\right\}\bigcup\left\{\bigcup_{l=0}^{k-1}R_{\theta_k^{-}(l)}\right\}
\label{eq:3.2}
\end{equation}
where for every $0 \leq l \leq k-1$, $\theta_k^{+}(l)$ and $\theta_k^{-}(l)$ are defined by \newline $\theta_k^{\pm}(l)=\frac{\pm\frac{\pi}{2}-\arg{a_k}}{k}+\frac{2\pi l}{k}$ and $\arg{a_k}$ is taken to be in $[0,2 \pi)$.
\end{lemma}

Observe that for every $0 \leq {l \neq j} \leq k-1$, $\theta_k^{\pm}(l) \neq \theta_j^{\pm}(l)$.
\begin{proof}
For every $z_0 \neq 0$ that is not in the union (\ref{eq:3.2}) there exist $r>0$ and $0 \leq l \leq k-1$ such that
\begin {equation}
\overline{\Delta}(z_0,r)\subset S\left(\frac{\theta_k^{+}(l)+\theta_k^{-}(l+1)}{2},\frac{\theta_k^{+}(l)-\theta_k^{-}(l+1)}{2}\right)
\label{eq:3.3}
\end{equation}
or that
\begin{equation}
\overline{\Delta}(z_0,r)\subset S\left(\frac{\theta_k^{-}(l)+\theta_k^{+}(l)}{2},\frac{\theta_k^{+}(l)-\theta_k^{-}(l)}{2}\right).
\label{eq:3.4}
\end{equation}
There is some small $\varepsilon_0>0$ such that in the case that (\ref{eq:3.3}) holds, then for every $z \in \Delta(z_0,r)$ and for every $n \in \mathbb{N}$
\begin{align*}
\pi/2+2\pi l +\varepsilon_0<\arg{a_k(nz)^k}<3\pi/2+2\pi l-\varepsilon_0.
\end{align*}
In the case (\ref{eq:3.4}), then for every $z\in \Delta(z_0,r)$
\begin{align*}
 -\pi/2+2\pi l +\varepsilon_0<\arg{(a_k(nz)^k)}<\pi/2+2\pi l-\varepsilon_0.
 \end{align*}

Hence there exists $N_0$, such that if $n>N_0$ and $z \in \Delta(z_0,r)$, then
\begin{align*}
\pi/2+2\pi l+\varepsilon_0/2<\arg{P(nz)}<3\pi/2+2\pi l-\varepsilon_0/2
\end{align*}
 in the case of (\ref{eq:3.4}).

Hence in the case of (\ref{eq:3.3}) $f(nz) \rightarrow 0$ uniformly in $\Delta(z_0,r)$ and in case of (\ref{eq:3.4}) $f(nz) \rightarrow \infty$ uniformly in $\Delta(z_0,r)$, that is, in any case $F(f)$ is normal at $z_0$.

If $z_0$ belongs to one of the $2k$ rays from the union (\ref{eq:3.2}), then any neighbourhood of $z_0$ contains points $z$ where $f(nz)\to 0$ and points $z$ where $f(nz) \rightarrow \infty$. So $F(f)$ is not normal at $z_0$.
\end{proof}
We are now ready to calculate $\Pi _\alpha{(f)}$. We shall do this by separating into 2 cases according to the value of $k=\left|P\right|$ .
\subsection{Calculating $\Pi _\alpha(Re^P)$ for linear polynomial $P(z)$.}
We have $P(z)=a_1z+a_0$, $a_1 \neq 0$. Let $z_0$ be a point where $F(f)$ is not normal. We assume thar for some $-1<\alpha<1$
\begin{equation}
f_{n,\alpha}(\zeta)=\frac{f(k_nz_n+k_n\rho_n\zeta)}{\rho_n^\alpha} \stackrel{\chi}{\Rightarrow}g(\zeta) \ ,
\label{eq:3.5}
\end{equation}
where $z_n \rightarrow z_0$, $\rho_n \rightarrow 0^+$ and $k_n \rightarrow \infty$.\newline
\textbf{Case (A)} $z_0 \neq 0$.

In this case
\begin{equation}
k_nz_n \rightarrow\infty \quad \text{and} \quad \frac{z_n}{\rho_n} \rightarrow \infty \ ,
\label{eq:3.6}
\end{equation}
and thus
\begin{align*}
\frac{R(k_nz_n+k_n\rho_n\zeta)}{(k_nz_n)^{L_1-L_2}} \Rightarrow 1  \ .
\end{align*}
We deduce that
\begin{equation}
\tilde{g}_n(\zeta):=(k_nz_n)^{L_1-l_2}\frac{e^{a_1k_nz_n+a_0}e^{a_1k_n\rho_n\zeta}}{\rho_n^{\alpha}} \Rightarrow g(\zeta) \ .
\label{eq:3.7}
\end{equation}
Since $\tilde{g}_n(\zeta)\neq 0 $ for $\zeta \in \mathbb {C}$, we deduce that $g \neq 0$ in $\mathbb{C}$, i.e., $g=e^{Q}$ where $Q$ is an entire function. With a suitable branch of the logarithm, we have
\begin{align*}
e^{a_1k_nz_n+a_0-\alpha \ln \rho_n+(L_1-L_2)\log k_nz_n+a_1k_n\rho_n\zeta} \Rightarrow e^{Q(\zeta)} \ .
\end{align*}
Thus, there are integers $m_n$ such that
\begin{align*}
 a_1k_nz_n+a_0-\alpha\ln \rho_n+(L_1-L_2)\log k_nz_n+a_1k_n\rho_n\zeta+2\pi i m_n \Rightarrow Q(\zeta)  \ .
 \end{align*}
Hence $Q$ is a linear function, $Q(\zeta)=A_1\zeta+A_0$ and $g(\zeta)=e^{A_0}\cdot e^{A_1 \zeta}$. Substituting $\zeta=0$ in (\ref{eq:3.7}) gives that
\begin{align*}
\frac{(k_nz_n)^{L_1-L_2}e^{a_1k_nz_n+a_0}}{\rho_n^\alpha} \mathop  \to \limits_{n \to \infty }  e^{A_0} \ ,
\end{align*}
and thus
\begin{equation}
a_1k_n\rho_n \to A_1
\label{eq:3.8}
\end{equation}
and $\arg A_1=\arg a_1$.
By (1) and (2) of Lemma 3.1 we deduce that the contribution of $z_0 \neq 0$, point of non-normality of $F(f)$ to $\Pi _\alpha(f)$, is
\begin{equation}
\left\{k_0e^{A_1\zeta}: k_0 \neq 0,\ \arg A_1=\arg a_1 \right\} \ .
\label{eq:3.9}
\end{equation}
Observe that this collection is independent of $\alpha$.\newline
\newline
\textbf{Case (B)} $z_0=0$.

\noindent We separate into subcases according to the behaviour of $\left\{k_nz_n\right\}$.

\noindent $\boldsymbol{k_nz_n \rightarrow b}$, $\boldsymbol{b \in \mathbb{C}}$.

In this case, if $k_n\rho_n \rightarrow \infty$, then when $\alpha \leq 0$ it holds for every $\zeta \neq 0$, $\zeta \in R_\theta$ and $\theta_1^{-}(0)<\theta<\theta_1^{+}(0)$, that $f_{n,\alpha}(\zeta) \mathop \to \limits_{n\rightarrow \infty} \infty$, and this implies that $g\equiv \infty$, a contradiction. If $\alpha \geq 0$, then for every $\zeta \neq 0$, $\zeta \in R_{\theta}$,
\begin{equation}
\theta_1^+(0)<\theta<\theta_1^-(1) \ ,
\label{eq:3.10}
\end{equation}
we have $f_{n,\alpha}(\zeta)\rightarrow 0$ and this also leads to a contradiction.

If $k_n\rho_n \rightarrow a$, $a>0$, then in case that $\alpha>0$, it holds for every $\zeta$ such that $R(a\zeta+b) \neq 0$ that $g(\zeta)=\infty$, and this is impossible.

If $\alpha<0$, then for every $\zeta$ such that $R(a\zeta+b)\neq \infty$, $g(\zeta)=0$, again a contradiction.

So the case $k_nz_n \rightarrow b$, $k_nz_n \rightarrow a>0$ can happen only with $\alpha=0$, and indeed in this case the limit function is $g(\zeta)=f(a\zeta+b)$ and every such function is attained with $k_n=n$, $\rho_n=\frac{a}{n}$, $z_n=\frac{b}{n}$.

So this possibility gives the collection
\begin{equation}
\left\{f(a\zeta+b): a>0, b \in \mathbb{C}\right\}
\label{eq:3.11}
\end{equation}
to $\Pi_0(f)$.
\vskip 10 pt
We are left with the option $k_n\rho_n \rightarrow 0$. We then have that
\begin{equation}
R_{n,\alpha}(\zeta)=\frac{R(k_nz_n+k_n\rho_n\zeta)}{\rho_n^\alpha}\stackrel{\chi}{\Rightarrow} g(\zeta)e^{-P(b_1)}.
\label{eq:3.12}
\end{equation}
If $\alpha=0$ then $g$ is a constant, a contradiction. If $0<\alpha<1$, then in the case that $P_1(z)$ is a constant, $R_{n,\alpha}(\zeta)\Rightarrow \infty$ and $g \equiv \infty$, a contradiction. If $P_1$ is not a constant then necessarily there exists some $1 \leq i \leq m$ such that $k_nz_n\mathop \to \limits_{n \rightarrow \infty}\gamma_i$. We then have
\begin{align*}
 \frac{(k_nz_n-\gamma_i+k_n\rho_n\zeta)^{l_i}}{\rho_n^\alpha}\Rightarrow\frac{g(\zeta)e^{-P(\gamma_i)}}{\tilde{R}_{\gamma_i}(\gamma_i)}  \ .
 \end{align*}
By the case of monome (see (\ref{eq:2.1.75})), we get that
\begin{equation}
g(\zeta)=e^{P(\gamma_i)}\tilde{R}_{\gamma_i}(\gamma_i)(a\zeta+b)^{l_i},\ b\in \mathbb{C},\ a>0
\label{eq:3.13}
\end{equation}
and by the setting of (\ref{eq:2.1.5}), every $g(\zeta)$ of the form (\ref{eq:3.13}) belongs to $\Pi _\alpha{(f)}$ (corresponding to all the roots $\gamma_i$ of $P_1(z)$, $1\leq i\leq m$).

Now, if $-1<\alpha<0$ then $0<-\alpha<1$ and as in Case (B) of (I) in section 2.3, or in (II) in section 2.3, we get that if $P_2(z)$ is a constant and then $g\equiv \infty$, a contradiction. If $P_2(z)$ is not a constant then $k_nz_n\mathop \to\limits_{n\to\infty}\beta_i$ for some $1\leq i \leq l$, and analogously to (\ref{eq:3.13}) we have
\begin{equation}
g(\zeta)=\frac{e^{P(\beta_i)}\hat{R}_{\beta_i}(\beta_i)}{(a\zeta+b)^{j_i}},\ a>0,\ b \in \mathbb{C} \ ,
\label{eq:3.14}
\end{equation}
and conversely, every function $g(\zeta)$ as in (\ref{eq:3.14}), (corresponding to the various roots of $P_2(z)$, $\beta_i$, $1\leq i\leq l$) belongs to $\Pi _\alpha{(f)}$.

\vskip 10 pt
We turn now to the second subcase of Case (B).
\vskip 10 pt

\noindent$\boldsymbol{k_nz_n \rightarrow \infty}$.

 In this situation, if $k_n\rho_n \rightarrow \infty$ and $ \frac{z_n}{\rho_n} \rightarrow \infty$ then (\ref{eq:3.5}) is equivalent to
 \begin{align*}
  \frac{(k_nz_n)^{l_1-l_2}}{(\rho_n^\alpha)}e^{a_1k_nz_n+a_0}e^{a_1k_n\rho_n\zeta}\Rightarrow g(\zeta)  \ ,
  \end{align*}
   and we deduce that we must have $g(\zeta)=k_0e^{a\zeta}$.

    On the other hand, for every $\zeta$, $\zeta \notin R_{\theta_1^+(0)}\bigcup R_{\theta_1^-(0)}$, $g(\zeta)=0$ or $g(\zeta)=\infty$, and this is a contradiction.

     Suppose that $k_n\rho_n \rightarrow \infty$ and $\frac{z_n}{\rho_n}\rightarrow d$, $d \in \mathbb{C}$. Then (\ref{eq:3.5}) can be written as
\begin{equation}
f_{n,\alpha}(\zeta)=\frac{R[k_n\rho_n(\zeta+\frac{z_n}{\rho_n})]e^{a_0}e^{a_1k_n\rho_n(\zeta+\frac{z_n}{\rho_n})}}{\rho_n^\alpha} \stackrel{\chi}{\Rightarrow} g(\zeta) \ .
\label{eq:3.14.5}
\end{equation}
When $\zeta$ belongs to the half plane $\left\{\zeta:-\pi/2<\arg(a_1)+\arg(\zeta+C)<\pi/2 \right\}$ we have $f_{n,\alpha}(\zeta) \rightarrow \infty$ if $\alpha \geq 0$, while if $\alpha \leq 0$, then $f_{n,\alpha}(\zeta)\rightarrow 0$ for every $\zeta$ in the complementary half plane, $\left\{\zeta:\pi/2< \arg(a_1)+\arg(\zeta+C)< 3\pi/2\right\}$, and we have got a contradiction.

To summarize, the possibility $k_nz_n \rightarrow \infty$ and $k_n\rho_n \rightarrow \infty$ does not occur.

Now if $k_n\rho_n \to a$, $a\in \mathbb{C}$ then (\ref{eq:3.5}) is equivalent to (\ref{eq:3.7}) and \newline $g(\zeta)=e^{A+B\zeta}$ and it must be that $a>0$ and $A=a\cdot a_1$.

In order to show that for each $B\in \mathbb{C}$ and for each $A$ satisfying $\arg(A)=\arg(a_1)$, the function $g(\zeta)=e^{A\zeta+B}$ belongs to $\Pi _\alpha{(f)}$, it is enough by Lemma 3.1 to show that one such function is attained (in fact, it is equally easy to show directly that each such function is attained).

Indeed, let us take a sequence of non-zero numbers, $z_0^{(l)}\mathop\to\limits_{l\to\infty}0$ such that for every $l \geq 1$, $\arg{(z_0^{(l)})}=\pi/2-\arg(a_1)$. By the results in Case (A) (see (\ref{eq:3.9})), for every $l \geq 1$ there are sequences, $k_m^{(l)}\mathop\to\limits_{m\to\infty}\infty$, $z_m^{(l)}\mathop\to\limits_{m\to\infty} z_0^{(l)}$ and $\rho_m^{(l)}\mathop\to\limits_{m\to\infty} 0^+$ such that
\begin{align*}
 \frac{f(k_m^{(l)}z_m^{(l)}+k_m^{(l)}\rho_m^{(l)}\zeta)}{\rho_m^{(l)\alpha}} \stackrel{\chi}{\Rightarrow} e^{a_1\zeta} \ .
 \end{align*}
Now for every $n\geq 1$, there is $m_n>n$ such that
\begin{equation}
\Big|k_{m_n}^{(n)}\cdot z_{m_n}^{(n)}\Big|>n,\quad \rho_{m_n}^{(n)}<\frac{1}{n} \quad \text{and} \quad \Big|z_{m_n}^{(n)}-z_0^{(n)}\Big|<\frac{1}{n}  \ ,
\label{eq:3.15}
\end{equation}

\noindent and such that
\begin{align*}
\max_{\zeta\leq n}\Big|\frac{f(k_{m_n}^{(n)}z_{m_n}^{(n)}+k_m^{(n)}\rho_{m_n}^{(n)}\zeta)}{\rho_{m}^{(n)^{\alpha}}}-e^{a_1\zeta}\Big|\leq \frac{1}{n} \ .
\end{align*}
We define now for every $n\geq 1$, $k_n:=k_{m_n}^{(n)}$, $\rho_n:=\rho_{m_n}^{(n)}$, $z_n:=z_{m_n}^{(n)}$. By (\ref{eq:3.15}) we deduce that
\begin{align*}
\frac{f(k_nz_n+k_n\rho_n\zeta)}{\rho_n^{\alpha}}\stackrel{\chi}{\Rightarrow} e^{a_1\zeta} \ ,
\end{align*}
as required (with $k_nz_n\to\infty$ and $k_n\rho_n\to 1$) see (\ref{eq:3.8}).

Hence the collection of limit functions created by the possibility \linebreak $k_nz_n\rightarrow \infty$ and $k_nz_n\to a$, $a\in \mathbb{C}$ is exactly
\begin{equation}
\left\{e^{A\zeta+B}:B\in\mathbb{C} \quad \text{and} \quad \arg{A}=\arg{(A_1)} \right\}.
\label{eq:3.16}
\end{equation}
We can now summarize the results and conclude the assertion of Theorem 1 for the case where $P$ is linear.

For $\alpha=0$, we get by (\ref{eq:3.9}), (\ref{eq:3.11}), and (\ref{eq:3.16}) (and the various contradictions along the way)
\begin{align*}
{\Pi }_0(f)=\left\{e^{a\zeta+b}:\arg{a}=\arg{a_1}, b\in\mathbb{C}\right\}\bigcup \left\{f(a\zeta+b):a>0,\ b\in\mathbb{C}\right\} \ .
\end{align*}
For $0<\alpha<1$, (\ref{eq:3.9}), (\ref{eq:3.13}) and (\ref{eq:3.16}) give
\begin{align*}
{\Pi }_\alpha(f)=\left\{e^{a\zeta+b}:\arg{a}=\arg{a_1},\ b\in\mathbb{C}\right\}
\end{align*}
\begin{align*}
\bigcup \left\{e^{P(\gamma_i)}\tilde{R}_{\gamma_i}(\gamma_i)(a\zeta+b)^{l_i}: a>0,\ b\in\mathbb{C},\ 1\leq i \leq m  \right\}.
\end{align*}

\noindent For $-1<\alpha<0$ we have by (\ref{eq:3.9}), (\ref{eq:3.14}) and (\ref{eq:3.16})
\begin{align*}
{\Pi }_\alpha(f)=\left\{e^{a\zeta+b}:\arg{a}=\arg{a_1},\ b\in\mathbb{C}\right\}
\end{align*}
\begin{align*}
\bigcup \left\{e^{P(\beta_i)}\hat{R}_{\beta_i}(\beta_i)/(a\zeta+b)^{j_i}: a>0, b\in\mathbb{C}, 1\leq i \leq l  \right\}.
\end{align*}
\subsection{Calculating \textbf{$\Pi _\alpha(Re^{P})$} when \textbf{$k=\left|P \right|\geq 2$}.}
We consider (\ref{eq:3.5}) and separate into cases according the behaviour of $\left\{k_nz_n\right\}$.

\noindent \textbf{Case (A)} \quad $\boldsymbol{k_nz_n\to b\in\mathbb{C}}$.

Of course in this case $z_n \to 0$.

If $k_n\rho_n \to\infty$, then if $\alpha\leq 0$ it holds for every non-zero $\zeta$, $\zeta\in R_\theta$, for $\theta_k^{-}(l)<\theta<\theta_k^{+}(l)$, $0\leq l \leq k-1$, that $f_{n,\alpha}(\zeta)\mathop{\to}\limits_{n\to\infty}\infty$ (compare (\ref{eq:3.10})), and this is a contradiction. If $\alpha \geq 0$ then for every non-zero $\zeta$, $\zeta\in R_\theta$, $\theta_k^{+}(l)<\theta<\theta_k^{-}(l+1)$, $f_{n,\alpha}(\zeta)\mathop{\to}\limits_{n\to\infty} 0$, and this is a contradiction.

Hence we deduce that $k_n\rho_n\to a \in \mathbb{C}$.

If $a>0$ and $\alpha \neq 0$, then similarly to the parallel case when \newline $\left|P\right|=k=1$ (Case (B) in section 3.2) we get a contradiction.

The possibility $a>0$ and $\alpha=0$, as in the case $\left|P\right|=1$, gives the collection
\begin{equation}
\left\{f(a\zeta+b):a>0,\ b\in \mathbb{C}\right\}
\label{eq:3.16.5}
\end{equation}
to $\Pi_0(f)$.\newline
We are left with the possibility $k_n\rho_n \to 0$. We then get that
\begin{align*}
\frac{R(k_nz_n+k_n\rho_n\zeta)}{\rho_n^\alpha}\stackrel{\chi}{\Rightarrow} g(\zeta)e^{-P(b)} \ ,
\end{align*}
that is, $\tilde{g}:=g\cdot e^{-P(b)}$ belongs to ${\Pi }_\alpha(R)$. Thus, in the case $0<\alpha<1$ we get by the discussion in section 2.3 that for some $1\leq i_0 \leq m$, $b=\gamma_{i_0}$ (in case $\big| P_1 \big|>0$, otherwise we get a contradiction) and consider all $\gamma_i$, $1 \leq i \leq m$, we get from (\ref{eq:2.8}) that the case $k_n\rho_n \to 0$, $k_nz_n\to b\in \mathbb{C}$ gives the collection
\begin{equation}
\label{eq:3.17}
\bigcup_{i=1}^{m}\left\{e^{P(\gamma_i)}\tilde{R}_{\gamma_i}(\gamma_i)(A_1\zeta+A_2)^{l_1},\ A_1>0,\ A_2\in \mathbb{C} \right\}
\end{equation}
to $\Pi_0(f)$.\newline
In the case $-1<\alpha<0$ we get (similarly to the parallel subcase in Case~(B) in Section 3.2) the collection
\begin{equation}
\bigcup_{i=1}^{l}\left\{e^{P(\beta_i)}\hat{R}_{\beta_i}(\beta_i)(A_1\zeta+A_2)^{-j_1},\ A_1>0,\ A_2\in \mathbb{C} \right\} \ .
\label{eq:3.18}
\end{equation}
The case $\alpha=0$ leads to a contradiction, similarly to the parallel case in Case~(B) in Section 3.2.

\noindent \textbf{Case (B)} \quad $\boldsymbol{k_nz_n \to \infty}$

We have $z_n\to z_0$, and in this case both options $z_0=0$ or $z_0 \neq 0$ are possible. First we deal with the option $z_0\neq 0$, i.e. $z_0=re^{i\theta_0}$ where  $\theta_0$ is one of the arguments of the $2k$ rays from (\ref{eq:3.2}). Since $\frac{\rho_n}{z_n}\to 0$, then (\ref{eq:3.5}) is equivalent to
\begin{equation}
(k_nz_n)^{L_1-L_2}\frac{e^{P(k_nz_n+k_n\rho_n\zeta)}}{\rho_n^\alpha} \stackrel {\chi}{\Rightarrow}g(\zeta) \ .
\label{eq:3.19}
\end{equation}
By Hurwitz's Theorem $g(\zeta)=e^{Q(\zeta)}$, where $Q$ is an entire function. For a suitable branch of the logarithm, we have
\begin{align*}
e^{P(k_nz_n+k_n\rho_n\zeta)+(L_1-L_2)\log{k_nz_n}-\alpha \ln{\rho_n}}\Rightarrow e^{Q(\zeta)} \ .
\end{align*}
Thus, there exist integers $\left\{m_n\right\}$ such that
\begin{equation}
{P(k_nz_n+k_n\rho_n\zeta)+(L_1-L_2)\ln{\left|k_nz_n\right|}+i(L_1-L_2)(\theta_0+\varepsilon_n)-\alpha \ln{\rho_n}+2\pi m_n}
\label{eq:3.20}
\end{equation}
\begin{align*}
\Rightarrow Q(\zeta) \ ,
\end{align*}
where $\varepsilon_n \in \mathbb{R}$, $\varepsilon_n \to 0$.

We conclude that $Q$ is a polynomial of degree $\left|Q\right|\leq k$. Denote $Q(z)=A_0+A_1\zeta+...+A_k\zeta^k$. Comparing coefficients of the two sides of (\ref{eq:3.20}) gives the following relations

\begin{equation}
\begin{array}{l}
 a_k (k_n \rho _n )^k  \mathop{\to}\limits_{n\to\infty} A_k  \\
 \,\,\,\,\,\,\,\,\,\,\,\,\,\, \vdots  \\
 a_k \left( {k_n \rho _n } \right)^{k - i} \left( \begin{array}{l}
 k \\
 k - i \\
 \end{array} \right)\left( {k_n z_n } \right)^i  \to A_{k - i}  \\
 \,\,\,\,\,\,\,\,\,\,\,\,\, \vdots  \\
 a_k \left( {k_n \rho _n } \right)\left( \begin{array}{l}
 k \\
 1 \\
 \end{array} \right)\left( {k_n z_n } \right)^{k - 1}  \to A_1  \\
 a_k (k_n z_n  - \alpha _1 )...(k_n z_n  - \alpha _k ) + \left( {L_1  - L_2 } \right)\ln \left| {k_n z_n } \right|\\
  +i\left( {L_1  - L_2 } \right)(\theta _0  + \varepsilon _n ) - \alpha \ln \rho _n  + 2\pi im_n  \to A_{_0 }  \\
 \end{array}
 \label{eq:3.21}
 \end{equation}

 Now, if $k_n\rho_n \rightarrow \infty$, then by the relation of $A_k$ in (\ref{eq:3.21}) we deduce that $A_k=\infty$, a contradiction. If $k_n\rho_n \to a$, $a>0$ then by the relation for $A_{k-1}$ in (\ref{eq:3.21}), we get that $A_{k-1}=\infty$ (here we use $k \geq 2$), a contradiction. Hence we deduce that $k_n\rho_n \to 0$. Then from (\ref{eq:3.21}), we see that if $A_i \neq 0$ for some $2 \leq i \leq k$, then $A_{i-1}=\infty$. Thus $A_i=0$ for $2\leq i \leq k$.

 We can assume that $A_1\neq 0$ (since otherwise $g$ is a constant function), and so $g(\zeta)=e^{A_0+A_1\zeta}$. By (\ref{eq:3.21}) and (\ref{eq:3.2}) we get
 \begin{equation}
 \label{eq:3.22}
 \arg{A_1}=\arg{a_k}+(k-1)\theta_0=\frac{(k-1)(\pm \pi/2)+\arg{a_k}+(k-1)2\pi l}{k}
 \end{equation}
 for some $0 \leq l \leq k-1$.

 We observe that in (\ref{eq:3.22}) there are $2k$ different arguments.

 By the fact that there must be some limit function $g$ and by Lemma~3.1, we obtain that the possibility $z_n\to z_0 \neq 0$ gives (for every $-1<\alpha<1$) the collection
 \begin{equation}
 \bigcup_{l=0}^{k-1}\left\{e^{A_0+A_1\zeta}:A_0\in \mathbb{C}, \ \arg{A_1}=\frac{\arg{a_k}+(k-1)(\pm \pi/2)+(k-1)2\pi l}{k} \right\}
 \label{eq:3.23}
\end{equation}
 to $\Pi_\alpha(f)$.

 We turn now to the case where $k_nz_n \rightarrow \infty$ and $z_n \to 0$.
 \vskip 10 pt
 \begin{claim} $k_n\rho_n \to 0$.
 \end{claim}
 \begin{proof}
 If $k_n\rho_n \to \infty$, then in the case that $\frac{z_n}{\rho_n}\to \infty$, (\ref{eq:3.19})-(\ref{eq:3.21}) hold and we get a contradiction by the relation for $A_k$ in (\ref{eq:3.21}). If $\frac{z_n}{\rho_n}\to b \in\mathbb{C}$, we get a contradiction similarly to the parallel case in section 3.2 (see (\ref{eq:3.14.5})).

 If on the other hand, $k_n\rho_n\to a$, $0<a<\infty$, then the relations in (\ref{eq:3.21}) hold, and by the relation of $A_{k-1}$ we get that $A_{k-1}=\infty$, a contradiction and the claim is proven.
 \end{proof}
 We can deduce now, as in the case where $z_n \to z_0\neq 0$, that \newline $g(\zeta)=e^{A_1\zeta+A_0}$ and for $A_1$, $A_0$ the two last relations in (\ref{eq:3.21}) hold,  respectively.

 We separate now according to the value of $\alpha$.
  \vskip 10 pt
 \noindent \textbf{Case (B1)} \quad $\boldsymbol{0<\alpha<1}$.\newline We can assume that $\arg(z_n) \to \theta_0$.
 \begin{claim} There is some $0\leq l \leq k-1$, such that $\pi/2+2\pi l\leq \arg{a_k}+k\theta_0 \leq 3\pi/2+2\pi l$.
 \end{claim}
 \begin{proof}
 If it is not the case, then we have $Re(P(k_nz_n))\to +\infty$.
 Without loss of generality, we can assume that $\left\{ \frac{\ln{\big|k_nz_n \big|}}{\ln{\rho_n}}\right\}$ converges (in the extended sense). Now, if
  $\quad \frac{\ln{\big|k_nz_n \big|}}{\ln{\rho_n}}\to b$, $0<b\leq \infty$, then $\frac{-\alpha\ln\rho_n}{Re[P(k_nz_n)]} \to 0$, and since $\frac{\ln\big|k_nz_n\big|}{Re[P(k_nz_n)]}\to 0$, we deduce that the real part of the left side of the relation for $A_0$ in (\ref{eq:3.21}) tends to $+ \infty$, and this is a contradiction. If on the other hand $\frac{\ln\big|k_nz_n\big|}{\ln\rho_n} \to 0$, then since $\alpha>0$ we derive the same conclusion and get a contradiction.

This completes the proof of the claim.
 \end{proof}
 Hence we can write
 \begin{equation}
 \frac{2\pi l}{k}+\frac {\pi/2-\arg{a_k}}{k} \leq \theta _0 \leq \frac{3\pi/2-\arg{a_k}}{k}+\frac{2\pi l}{k}
 \label {eq:3.24}
 \end{equation}
 for some $0\leq l \leq k-1$.

 We denote $\theta_1:=\arg{A_1}$, and by the relation for $A_1$ in (\ref{eq:3.21}) we have $\theta_1=\arg{a_k}+(k-1)\theta_0$ and thus
 \begin{equation}
 \arg{a_k}+\frac{k-1}{k}\left(\frac{\pi}{2}-\arg{a_k}+2\pi l \right)\leq \theta_1 \leq \arg{a_k}+\frac{k-1}{k}\left(\frac{3\pi}{2}-\arg{a_k}+2\pi l\right) \ ,
 \label{eq:3.25}
 \end{equation}
 $ 0\leq l \leq k-1$.

 We show now that for every $\theta_1$ that satisfies (\ref{eq:3.25}), there is $g \in {\Pi }_\alpha(f)$, $g(\zeta)=e^{A_0+A_1\zeta}$ with $\arg{A_1}=\theta_1$.

 Evidently it is enough for this purpose to show that for every $\theta_0$ that satisfies (\ref{eq:3.24}), there are sequences $\left\{k_n\right\}$, $k_n \in \mathbb{N}$, $\left\{m_n\right\}$, $m_n\in \mathbb{Z}$ and $\left\{z_n\right\}$, $\left\{\rho_n\right\}$, $z_n \to 0$ with $\arg{z_n} \to \theta$, $\rho_n \to 0^+$, such that the relations (\ref{eq:3.21}) hold (with $0=A_2=\dots=A_k$, $A_1 \neq 0$ and $A_0 \in \mathbb{C} $ arbitrary).

 We first show it for $\theta_0$ that satisfies (\ref{eq:3.24}) with sharp inequalities (and the corresponding $\theta_1$ will satisfy (\ref{eq:3.25}) with sharp inequalities).

 Indeed,  for $n \geq 2$ define $k_n=n$, $\rho_n=\frac{1}{n^{1+\frac{k-1}{k}\frac{\ln\ln{n}}{\ln n}}}$ and \newline $\hat{z}_n=e^{i\theta_0}(\frac{-\ln{\rho_n}}{n})^{\frac{1}{k}}$.

 Observe that since $k \geq 2$, $k_n\rho_n \to 0$ and $k_n\hat{z}_n \to \infty$, we have
 \begin{align*}
 k_n\rho_n(k_n\hat{z}_n)^{k-1}=\frac{1}{n^{\frac{k-1}{k}\frac{\ln\ln n}{\ln n}}}\big[\big(1+\frac{k-1}{k}\frac{\ln\ln n}{\ln n}\big)\ln n \big]^{\frac{k-1}{k}}
 \end{align*}
 \begin{align*}
 =(1+\frac{k-1}{k}\frac{\ln\ln n}{\ln n})^\frac{k-1}{k} \mathop{\to}\limits_{n\to\infty} 1.
 \end{align*}
 In addition we have $\frac{\big|k_n\hat{z}_n\big|^k}{-\ln{\rho_n}}=1$.
 By the choice of $\theta_0$ (see (\ref{eq:3.24})), we get
 \begin{align*}
 \frac{Re[P(n\hat{z_n})+(L_1-L_2)\ln{|n\hat{z_n}}|]}{|a_k(n\hat{z}_n)^k|}\mathop{\to}\limits_{n\to\infty} \cos(\arg{a_k}+k\theta_0)<0 \ ,
 \end{align*}
  and then we get
  \begin{equation}
\frac{-\alpha\ln{\rho_n}}{-Re[P(n\hat{z}_n)+(L_1-L_2)\ln{|n\hat{z}_n|}]} \to \frac{-\alpha}{|a_k| \cos(\arg{a_k}+k\theta_0)}>0.
  \label{eq:3.26}
  \end{equation}

  Denote $C_0=\frac{-\alpha}{|a_k| \cos(\arg{a_k}+k\theta_0)}$. From (\ref{eq:3.26}) we deduce that for large enough $n$
  \begin{align*}
  \frac{C_0}{2}[-Re[P(n\hat{z}_n)+(L_1-L_2)\ln{|n\hat{z}_n|}]]<-\alpha \ln{\rho_n}
  \end{align*}
  \begin{align*}
  < 2C_0[-Re(P(n\hat{z}_n))+(L_1-L_2)\ln{|n\hat{z}_n|}].
  \end{align*}

  By the Mean Value Theorem there is some $t_n$, $\sqrt[k]{\frac{C_0}{2}}<t_n<\sqrt[k]{2C_0}$ such that
  \begin{equation}
  \label{eq:3.27}
  -Re[P(n\hat{z}_nt_n)+(L_1-L_2)\ln{|n\hat{z}_nt_n|}]=-\alpha \ln{\rho_n}.
  \end{equation}
 ( In fact, it is easy to see that every sequence $\left\{t_n\right\}$ of real numbers that satisfies (\ref{eq:3.27}) must satisfy $t_n \to\sqrt[k]{C_0}$.)

  We set $z_n=t_n\hat{z}_n$ and then the relation for $A_1$ in (\ref{eq:3.21}) holds for some $A_1$ with $\arg{A_1}=\arg{a_k}+(k-1)\theta_0$. By (\ref{eq:3.27}) there are (after moving to subsequence if necessary, that will be denoted with no loss of generality with the same indices) integers $m_n$, $\boldsymbol{n \geq 2}$ such that the relation with regard to $A_0$ in (\ref{eq:3.21}) holds for some $A_0 \in \mathbb{C}$.

  Morever, since $k_n\rho_n \to 0$ and $k_nz_n\to \infty$, we deduce that the relations for $A_2,\dots,A_k$ in (\ref{eq:3.21}) hold and give $0=A_2=A_3=\dots=A_k$.

  The fulfillment of these relations in (\ref{eq:3.21}) means that
  \begin{align*}
  \frac{f(k_nz_n+k_n\rho_n\zeta)}{\rho_n^\alpha}\stackrel{\chi}{\Rightarrow}e^{A_0+A_1\zeta} \ .
  \end{align*}

  By (2) of Lemma 3.1, every function $g$, $g(\zeta)=e^{a\zeta+b}$ with \newline $\arg{a}=\arg{a_k}+(k-1)\theta_0=\theta_1$, and arbitrary $b\in \mathbb{C}$ is in ${\Pi }_\alpha(f)$.

  Now suppose that $\theta_1$ is equal to the left or to the right side of (\ref{eq:3.25}). Without loss of generality,
  \begin{align*}
  \theta_1=\arg{a_k}+\frac{k-1}{k}\left[\frac{3\pi}{2}-\arg{a_k}+2\pi l \right] \ , \ 0 \leq l \leq k-1 \ .
  \end{align*}
   Then we take an increasing sequence, $\left\{\theta_1^{(l)}\right\}_{l=1}^\infty$ such that
  \begin{align*}
  \arg{a_k}+\frac{k-1}{k}\left(\frac{\pi}{2}-\arg{a_k}+2\pi l \right)< \theta_1^{(l)}\mathop{\nearrow}\limits_{l \to \infty} \theta_1.
  \end{align*}
  By the case of sharp inequality in (\ref{eq:3.25}), for every $l \geq 1$, correspond sequences $z_n^{(l)}\mathop{\to}\limits_{n\to\infty} 0$, $\rho_n^{(l)}\mathop{\to}\limits_{n\to\infty}0^+$ such that
  \begin{align*}
  \frac{f(z_n^{(l)}+n\rho_n^{(l)}\zeta)}{\rho_n^{(l)\alpha}}\mathop{\Rightarrow}\limits_{n\to\infty}^{\chi}e^{{e^{i\theta_1^{(l)}}}\zeta} \ .
  \end{align*}
  Since
  \begin{align*}
  e^{e^{i\theta_1(l)}\zeta} \mathop{\Rightarrow}\limits_{l\to\infty} e^{e^{i\theta_1}\zeta} \ ,
  \end{align*}
   then in a similiar way to the case $k_nz_n \to\infty$, \ $k_n\rho_n\to a$ in Case (B) of section 3.2, we deduce the existence of sequences $\rho_n \to 0^+$, $z_n \to 0$, and $\left\{k_n\right\}$ such that
  \begin{align*}
  \frac{f(k_nz_n+k_n\rho_n\zeta)}{\rho_n^\alpha}\mathop{\Rightarrow}\limits_{n\to\infty} e^{e^{i\theta_1}\zeta}.
 \end{align*}
 As usual, by Lemma 3.1 every $g(\zeta)=e^{a\zeta+b}$ with $\arg{a}=\theta_1$ and arbitrary $b \in \mathbb{C}$ belongs to ${\Pi }_\alpha(f)$.

 In order to determine explicitly ${\Pi }_\alpha(f)$, we need to find the range of $\theta_1$ in (\ref{eq:3.25}).
 For $k=2$ we have
 \begin {equation}
 \label{eq:3.28}
 l=0: \quad \frac{\pi}{4}+\frac{\arg {a_2}}{2} \leq \theta_1 \leq \frac{3\pi}{4}+\frac{\arg{a_2}}{2}
 \end{equation}
 \begin{align*}
 l=1: \quad \frac{5\pi}{4}+\frac{\arg{a_2}}{2} \leq \theta_1 \leq\frac{7\pi}{4}+\frac{\arg{a_2}}{2}.
 \end{align*}
 There are two distinct intervals with sum of length $\pi$.
 \begin{claim} For $k\geq 3$ the range of $\theta_1$ in (\ref{eq:3.25}) is $[0,2\pi]$.
 \end{claim}
 \begin{proof}
 Denote for $0 \leq l \leq k-1$, the general interval in (\ref{eq:3.25}) by \linebreak $I_l=[\varepsilon_l,\delta_l]$. The length of $I_l$ is $|I_l|=\pi\frac{k-1}{k}$ and $\varepsilon_{l+1}-2\pi+\frac{2\pi}{k}=\varepsilon_l$. Thus it is enough to show that $\frac{k-1}{k}\pi\geq\frac{2\pi}{k}$ and $\frac{k-1}{k}\pi+(k-1)\frac{2\pi}{k}\geq 2\pi$. It is easy to see that these two inequalities are satisfied for $k \geq 3$. The claim is proven.
 \end{proof}

 As a result, from the claim and from Lemma 3.1, we get that for $k \geq 3$ the posiibility $z_n\to 0$, $k_nz_n\to\infty$ gives the collection (for $0<\alpha<1$)
 \begin {equation}
 \label{eq:3.29}
 \left\{e^{a\zeta+b}: A\neq 0,\ b\in \mathbb{C}\right\}
 \end{equation}
 to $\Pi_\alpha(f)$.

 We turn now to the complemertary case.\newline
\vskip 10 pt
 \textbf{Case (B2)} \quad$\boldsymbol{k_nz_n\to \infty}$, $\boldsymbol{z_n \to 0}$, $\boldsymbol{-1<\alpha<0}$ \newline
 Here again, as $f_{n,\alpha}(\zeta)\stackrel {\chi}{\Rightarrow }g(\zeta)$ if and only if $(\frac{1}{f})_{n,-\alpha}(\zeta)\stackrel{\chi}{\Rightarrow}(\frac{1}{g})(\zeta)$ and since $\frac{1}{f}=\frac{1}{R}e^{-P}$, i.e., a function of the same type we get the following.

 For $k=2$, observe that $\frac{1}{e^{a\zeta+b}}=e^{-a\zeta-b}$ and $\arg{(-a)}=\pi+\arg{a}$ and also the leading coefficient of $-P(z)$ has the argument \newline $\arg(-a_2)=\pi+\arg{(a_2)}$. So we substitute in (\ref{eq:3.25}) (or in (\ref{eq:3.28})) these values (or $\arg{(a)}-\pi$ and $\arg{(a_2)}-\pi$ resp.) instead of $\theta_1$ and $\arg{(a_2)}$, respectively, to get
 \begin{equation}
 \label{eq:3.30}
\frac{\arg{a_2}}{2}+\frac{3\pi}{4}\leq \theta_1 \leq \frac{5\pi}{4}+\frac{\arg{a_2}}{2} \quad or \quad \frac{7\pi}{4}+\frac{\arg{a_2}}{2}\leq\theta_1 \leq \frac{9\pi}{4}+\frac{\arg{a_2}}{2}.
 \end{equation}
 Observe that the set of values of $a\in\mathbb{C}$ correesponds to (\ref{eq:3.30}) is the complement (up to the boundary) of the set of values of $a\in\mathbb{C}$ corresponding to (\ref{eq:3.28}).

 For $k \geq 3$ we get the collection
 \begin{equation}
 \label{eq:3.31}
 \left\{e^{a\zeta+b}:a\neq 0,b\in\mathbb{C}\right\}
 \end{equation}
to $\Pi _\alpha(f)$, exactly as in (\ref{eq:3.29}).

 The last case to treat is \newline
 $\boldsymbol{k_nz_n\to\infty}$, $\boldsymbol{z_n\to 0$, $\alpha=0}$.

 In this case as we saw also, $k_n\rho_n\to 0$.
 Also the relations in (\ref{eq:3.21}) hold and $A_i=0$ for $2\leq i \leq k$, and $A_1 \neq 0$.

 We can assume, without loss of generality, that $\arg{(z_n)}\to \theta_0$. From the relations for $A_0$ in (\ref{eq:3.21}), we get
 \begin{equation}
 \label{eq:3.32}
 \arg{a_k}+k\theta_0=\pm\frac{\pi}{2}+2\pi l \quad \text{for some} \ l\in\mathbb{Z} \ .
 \end{equation}
  And by the relation for $A_1$ in (\ref{eq:3.21}), we get
 \begin{equation}
 \label{eq:3.33}
 \theta_1:=\arg{A_1}=\arg{a_k}+(k-1)\theta_0=\frac{(k-1)(\pm\frac{\pi}{2})+\arg{a_k}+(k-1)2\pi l}{k},
 \end{equation}
 $0\leq l \leq k-1$.

 In the other direction we show now that every function of the form $g(\zeta)=e^{a\zeta+b}$, with $\theta_1=\arg{(a)}$, that satisfies (\ref{eq:3.33}) is obtained in ${\Pi }_0(f)$.

 Indeed, set $\theta_1=\theta_1(\theta_0)=\arg{a_k}+(k-1)\theta_0$.

 For every $\theta_0$ that satisfies (\ref{eq:3.32}) and for every $m \geq 1$, there exist according to (\ref{eq:3.23}) sequences $z_n^{(m)}\mathop{\to}\limits_{n \to\infty}\frac{1}{m}e^{i\theta_0}$, $\rho_n^{(m)}\mathop{\to}\limits_{n\to\infty} 0^+$ and $\left\{k_n^{(m)}\right\}_{n=1}^\infty$ such that
 \begin{align*}
  f(k_n^{(m)}z_n^{(m)}+k_n^{(m)}\rho_n^{(m)}\zeta)\mathop {\Rightarrow}\limits_{n\to\infty}^{\chi} e^{e^{i\theta_0}\zeta} \ .
 \end{align*}
 Hence, we get as in the case $k_nz_n\to\infty$ in Case (B) in Section 3.2 that \linebreak $g(\zeta)=e^{e^{i\theta_0}\zeta}$ is attained as a limit function with $z_n\to 0$ (and $\arg{z_n}=\theta_0$) and $\rho_n\to 0^+$. Then as usual by Lemma~3.1, we obtain that every $g(\zeta)=e^{a\zeta+b}$, with $\arg{a}=\theta_1$ where $\theta_1$ is as in (\ref{eq:3.33}). Thus this option gives the collection
 \begin{equation}
\bigcup_{l=0}^{k-1} \left\{e^{a\zeta+b}:b\in\mathbb{C},\arg{a}=((k-1)(\pm)\frac{\pi}{2}+\arg{a_k}+(k-1)2\pi l)/k\right\}
 \label{eq:3.34}
 \end{equation}

 to $\Pi_0(f)$.

 Observe that not as in the cases $0<\alpha<1$, $-1<\alpha<0$, this case does not add to ${\Pi _\alpha(f)}$, (here $\alpha=0$), new functions.

 Now we can finally collect all the limit functions to fix ${\Pi }_\alpha(f)$ for \linebreak$-1<\alpha<1$ in the case $k \geq 2$.

 \noindent$\boldsymbol{\alpha=0}$\newline
 For every $k\geq 2$ we get by (\ref{eq:3.16.5}), (\ref{eq:3.23}) (and (\ref{eq:3.34}))
 \begin{align*}
 {\Pi }_0(f)=\left\{f(a\zeta+b): a>0, b\in\mathbb{C}\right\} \bigcup
 \end{align*}
 \begin{align*}
 \left\{e^{a\zeta+b}:b\in\mathbb{C},\arg{a}=\frac{\arg{a_k}+(k-1)(\pm\frac{\pi}{2})+(k-1)2\pi l}{k},\quad 0\leq l \leq k-1 \right\} \ .
 \end{align*}\newline
 $\boldsymbol{0< \alpha <1}$ \newline
 For $k=2$ we get by (\ref{eq:3.17}) \textbf{and} (\ref{eq:3.23}) and (\ref{eq:3.28})
 \begin{align*}
 {\Pi }_\alpha(f)=\left\{\bigcup_{i=1}^m \left\{e^{P(\gamma_i)}\tilde{R}_{\gamma_i}(\gamma_i)(a\zeta+b)^{l_i}:a>0,\quad b\in\mathbb{C}\right\}\right\} \bigcup
 \end{align*}
 \begin{align*}
 \left\{e^{a\zeta+b}:\ b\in\mathbb{C},\ \frac{\pi}{4}+\frac{\arg{a_2}}{2}\leq \arg{a} \leq \frac{3\pi}{4}+\frac{\arg{a_2}}{2} \quad \text{or} \quad \right.
 \end{align*}
 \begin{align*}
  \left.\frac{5\pi}{4}+\frac{\arg{a_2}}{2}\leq \arg{a} \leq \frac{7\pi}{4}+\frac{\arg{a_2}}{2}\right\} \ .
 \end{align*}
 For $k \geq 3$ we get by (\ref{eq:3.17}), (\ref{eq:3.23}) \textbf{and} (\ref{eq:3.29})
 \begin{align*}
 {\Pi }_\alpha(f)=\Big[\bigcup_{i=1}^m\left\{e^{P(\gamma_i)}\tilde{R}_{\gamma_i}(\gamma_i)(a\zeta+b)^{l_i}: a>0, b\in \mathbb{C}\right\}\Big] \bigcup
 \end{align*}
 \begin{align*}
 \left\{e^{a\zeta+b}:a\neq 0, b\in \mathbb{C}\right\}.
 \end{align*}
 \noindent$\boldsymbol{-1<\alpha<0}$\newline
 For $k=2$ we get by (\ref{eq:3.18}), (\ref{eq:3.23}) \textbf{and} (\ref{eq:3.30})
 \begin{align*}
 {\Pi }_\alpha(f)=\Big[\bigcup_{i=1}^l \left\{e^{P(\beta_i)}\hat{R}_{\beta_i}(\beta_i)(a\zeta+b)^{-j_i}:a>0, \ b\in\mathbb{C}\right\}\Big]\bigcup
 \end{align*}
 \begin{align*}
 \left\{e^{a\zeta+b}:b\in\mathbb{C}, \frac{-\pi}{4}+\frac{\arg{a_2}}{2}\leq \arg{a} \leq \frac{\pi}{4}+\frac{\arg{a_2}}{2}\right. \quad \text{or} \quad
 \end{align*}
 \begin{align*}
  \left.\frac{3\pi}{4}+\frac{\arg{a_2}}{2}\leq \arg{a} \leq \frac{5\pi}{4}+\frac{\arg{a_2}}{2}\right\} \ .
 \end{align*}

 For $k \geq 3$ (\ref{eq:3.18}), (\ref{eq:3.23}) \textbf{and} (\ref{eq:3.31}) give
 \begin{align*}
 {\Pi }_\alpha(f)=\Big[\bigcup_{i=1}^{l}\left\{e^{P(\beta_i)}\hat{R}_{\beta_i}(\beta_i)(a\zeta+b)^{-j_i}:a>0,b\in\mathbb{C}\right\}\Big]\bigcup
 \end{align*}
 \begin{align*}
 \left\{e^{a\zeta+b}:a\neq 0 , b\in \mathbb{C}\right\}.
 \end{align*}

 The proof of Theorem 1 is completed.

\end {document}